\documentclass[11pt, a4paper]{amsart}
\usepackage{amsmath}
\newtheorem{theorem}{Theorem}[section]

\newtheorem{lemma}{Lemma}[section]

\newtheorem{corollary}{Corollary}[section]

\usepackage[colorlinks=true,citecolor=cyan,backref=page]{hyperref}

\usepackage{shuffle}

\newcommand{\tr}{\bigtriangleup}
\newcommand{\Z}{\mathbb Z}

\def\l{\langle}
\def\r{\rangle}

\usepackage{tikz}
\usetikzlibrary{patterns}
\usetikzlibrary{fadings}
\tikzfading[name=fade out, inner color=transparent!0, outer color=transparent!90]
\usepackage[labelfont=normalfont,labelformat=simple]{subcaption}
\usepackage{graphicx}
\usepackage[vcentermath]{youngtab}
\usepackage{ytableau}

\usepackage{enumerate}

\title{Primitive elements of the Hopf algebras of tableaux}

\author{C. Malvenuto}
\address{Claudia Malvenuto,
Dipartimento di Matematica ``Guido Castelnuovo'', Sapienza Universit\`a di Roma}
\email{claudia@mat.uniroma1.it}

\author{C. Reutenauer}
\address{Christophe Reutenauer,
D\'epartement de math\'ematiques, Universit\'e du Qu\'ebec \`a Montr\'eal}
\email{Reutenauer.Christophe@uqam.ca}

\date{\today}

\begin{document}

\maketitle

\tableofcontents

\begin{abstract} 
The character theory of symmetric groups, and the theory of symmetric functions, both make use of the combinatorics of Young tableaux, such as  the Robinson-Schensted algorithm, Sch\"utzenberger's {\em ``jeu de taquin''}, and evacuation.  In 1995 Poirier and the second author introduced  some algebraic structures, different from the plactic monoid, which induce some products and coproducts of tableaux, with homomorphisms. Their starting point are the two dual Hopf algebras of permutations, introduced by the authors in 1995.  In 2006 Aguiar and Sottile studied in more detail the Hopf algebra of permutations: among other things, they introduce a new basis, by M\"obius inversion in the poset of weak  order, that allows them to describe the primitive elements of the Hopf algebra of permutations. In the present Note, by a similar method, we determine the primitive elements of the Poirier-Reutenauer algebra of tableaux, using a partial order on tableaux defined by Taskin. 
\end{abstract}

\section{Introduction}
In 1995 the authors of the present paper introduced two dual Hopf algebra structures on permutations  \cite{MR}. The products and coproducts of permutations originated from the concatenation Hopf algebra and shuffle Hopf algebra on  $\mathbb Z\langle \mathbb N^{>0} \rangle$, the module generated by words of positive integers, from Solomon's descent algebra \cite{Sol} and Gessel's (internal) coalgebra \cite{G} of quasi--symmetric functions. The two Hopf structures on $\mathbb Z S$, the module with $\mathbb Z$--basis all the permutations in $S=\cup_{n\geq 0} S_n$, for $S_n$ the symmetric group on  $\{1,\ldots,n\}$, are autodual \cite{BS}. 

In the same year \cite{PR}, carrying on these themes, Poirier and the second author proved that the two dual Hopf structures on $\mathbb Z S$ are free associative algebras. By restriction on the plactic classes they obtained two dual structures of Hopf algebras on the $\mathbb Z$--module $\mathbb Z S$ with basis the set $T$ of all standard Young tableaux. The product and coproduct are described there in term of Sch\"utzenberger's {\em ``jeu de taquin''} \cite{LS}. They also provided different morphisms between these structures and the descent algebras, symmetric functions and quasi-symmetric functions. In particular, the map sending a permutation into its left tableau in the Schensted correspondence is a Hopf morphism. 

Loday and Ronco \cite{LR} characterized the product of two permutations by the use of the weak order of permutations: it is the sum of all permutations in some interval for this order. In 2005 \cite{AS}, Aguiar and Sottile studied in further and thorough details the structure of  the Hopf algebras of permutations, giving explicit formulas for its antipode, proving that it is a cofree algebra and determining its primitive elements. For the latter task, they introduced a new basis of $\mathbb Z S$, related to the basis of permutations via M\"obius inversion in the poset of the Bruhat weak order of the symmetric groups. In \cite{DHNT}, Duchamp, Hivert, Novelli and Thibon studied the Hopf algebra of permutations (denoted there $FQSYM$), and  gave among others a faithful representation by noncommutative polynomials.

The Hopf algebra of tableaux was used by J\"ollenbeck \cite{J}, and Blessenohl and Schocker \cite{BS}, to define their noncommutative character theory of the symmetric group. Moreover, Muge Taskin \cite{T} used the order on tableaux, induced by the weak order of permutations, to characterize the product of two tableaux, in a way reminiscent of the result of Loday and Ronco.

The purpose of this Note is to find the primitive elements of $\mathbb Z T$, the Hopf algebra of tableaux with respect to the product and coproduct, following the approach of Aguiar and Sottile. A new basis for $\mathbb Z T$ is obtained by M\"obius inversion for the Taskin order of tableaux. The nice feature of the proofs here is that we manage to avoid {\em ``jeu de taquin''}, and use a simpler description through a shifted left and right concatenation product of tableaux. 

\section{Preliminaries on permutations}

We denote by $S_n$ symmetric group on $\{1,\ldots,n\}$. We often represent permutations as words: $\sigma\in S_n$ 
is represented as the word $\sigma(1)\sigma(2)\cdots\sigma(n)$. By abuse of notation, we identify $\sigma$ and the 
corresponding word. A {\em word} in the sequel will always be on the alphabet of positive integers, also called {\em 
letters}. We denote by $|\sigma|$ the number of letters of $\sigma$.

We denote by $\leq$ the {\em right weak order} of permutations. Recall that it is defined as the reflexive and 
transitive closure of the relation $u < v$, $u,v\in S_n$, $v=u\circ\tau$, for some adjacent transposition $\tau\in S_n$  
such that 
$l(u)<l(v)$, where $l(u)$ denotes as usual the  {\em length} of $u$ in the sense of Coxeter groups. Recall that this order may also be defined by comparing inversions sets: let $Inv(\sigma)$ be the set 
of inversions (`by values') of $\sigma$, that is, the set of pairs $(j,i)$ with $j>i$ and $\sigma^{-1}(j)<\sigma^{-1}(i)$; then 
$u\leq v$ 
if and only if $Inv(u)\subseteq  Inv(v)$. Note that $(j,i)$ is an inversion of $\sigma$ if and only if $j>i$ and $j$ appears on 
the left of $i$ in the word representing $\sigma$. This definition applies also to any word.

Given $\sigma \in S_n$, and a subset $I$ of $[n]$, $\sigma\mid I$ denotes the word obtained by removing in the word 
$\sigma$ the digits not in $I$ (whereas the restriction of $\sigma$ to $I$   is the subword $\sigma\mid_I$ of the images of the letters in $I$). For example, for $\sigma=2517643$, one has  $2517643\mid\{2,3,6\}=263$ (while $\sigma\mid_{\{2,3,6\}}=514$).

Moreover, $v$ being a word without repetition of letters, $st(v)$ denotes the {\em standard permutation} of the word 
$v$, obtained by replacing each letter in $v$ by its image under the unique increasing bijection form the set of letters in 
$v$ onto $\{1,2,\ldots,|v|\}$. For example, $st(5713)=3412$.

\begin{lemma}\label{restr-order-perm}
If $s\leq t$ are permutations in $S_n$  and $I$ an interval in $\{1,\ldots,n\}$, then $st(s\mid I)\leq st(t\mid I)$.
\end{lemma}

\begin{proof}
Let $(j,i)$ be an inversion of $s\mid I$; then $i,j\in I$. Then $(j,i)$ is an inversion of $s$, hence of $t$. It follows that 
it is also an inversion of $t\mid I$.
Since standardizing amounts to apply an increasing bijection, we deduce the lemma.
\end{proof}

Let $S=\bigcup_{n\geq 0} S_n$ be the disjoint union of all symmetric groups. There is a classical associative product on $S$, denoted by $\square$, which turns it into a free monoid \cite{PR}: let $u\in S_p$ and $v\in S_q$; let $\bar v$ be obtained by adding $p$ to each digit in $v$; then $u\square v$ is the concatenation $u\bar v$ of $u$ and $\bar v$ (or {\em right shifted concatenation}). For example, $231\square 12=23145$, with here $p=3$. The free generators of this free monoid are the {\em indecomposable} permutations, which have some importance in algebraic combinatorics; see \cite{C}.

A variant of this product is as follows: given two permutations as above, $v\bigtriangleup u=\bar vu$ (which we refer to as  {\em left shifted concatenation}). Example: $12\bigtriangleup 231=45231$. 

Clearly, the product $\square$ and the opposite product of $\tr$ are conjugate under the mapping $w\mapsto \tilde w$ which reverses words:
$$
v\tr u=\widetilde{\tilde u\square \tilde v}.
$$ 
It follows that  $S$ with the product $\tr$ is also a free monoid, freely generated by the 
permutations which are indecomposable for this product, which are the reversals of the indecomposable permutations.

For later use, we need 

\begin{lemma}\label{perm-comp} The weak order $\leq$ on permutations is compatible with the product $\tr$: $u\leq u', v\leq v'\Rightarrow v\tr u\leq v'\tr u'$.
\end{lemma}

\begin{proof}
Let $u\in S_p$, $v\in S_q$. Suppose that $u\leq u'$, $v\leq v'$. We show that $v\tr u\leq v'\tr u'$. Let $(j,i)$ be an 
inversion of $v\tr u=\bar vu$. If $i,j\leq p$, then by construction of the product $\tr$, $(j,i)$ is an inversion of $u$, hence of 
$u'$ (since $u\leq u'$), hence of $v'\tr 
u'$. If $i,j> p$, then $(j-p,i-p)$ is an inversion of $v$, hence of $v'$ (since $v\leq v'$) and therefore $(j,i)$ is an 
inversion of $v'\tr u'$. If $i\leq p$ and $j>p$, then $(j,i)$ is an inversion of $v'\tr u'$. There is no other case, since $i<j$.
\end{proof} 

\section{Hopf algebra of permutations}\label{ZS}

Denote by $\mathbb Z S$ be the free $\mathbb Z$-module with basis $S$. We define on $\mathbb Z S$ a product, denoted by $*$ (called {\em destandardized 
concatenation}),
and a coproduct (called {\em standardized unshuffling}), denoted by $\delta$, which turn it into a Hopf algebra (see \cite{MR}). If $
\alpha\in S_p$, $\beta\in S_q$, $\alpha * \beta$ is the sum of all permutations in $S_{p+q}$ of the form $uv$ 
(concatenation of $u$ and $v$), where $u,v$ are of respective lengths $p,q$ and $st(u)=\alpha$, 
$st(v)=\beta$; for example, $12*21=1243+1342+1432+2341+2431+3421$. Moreover, for $\sigma\in S_n$, 
$$
\delta(\sigma)=\sum_{0\leq i\leq n} \sigma\mid\{1,\ldots,i\}\otimes st( \sigma\mid\{i+1,\ldots,n\}).$$ 

%numbering from 1 to the length $|v|$ of $v$ the letters of $v$, starting from the smallest.
An example of coproduct is
$\delta(3124)=\epsilon\otimes st(3124)+1\otimes st(324)+12\otimes st(34)+312\otimes st(4)+3124\otimes 
\epsilon=\epsilon\otimes 3124 + 1\otimes 213+12\otimes 12+312\otimes 1+3124\otimes \epsilon$. Here $\epsilon$ is the empty permutation in $S_0$, the neutral element of the bialgebra $\mathbb Z S$.

Following Aguiar and Sottile \cite{AS},
%\footnote{Aguiar and Sottile deal with the isomorphic dual Hopf algebra, the isomorphism being $\sigma\mapsto\sigma^{-1}$}
we define a new linear basis $\mathcal M_\sigma$ of $\mathbb Z S$, indexed by permutations, and called {\em monomial basis}. (Notice that Aguiar and Sottile deal with the isomorphic dual Hopf algebra, the isomorphism being $\sigma\mapsto\sigma^{-1}$, and also that they use the left weak order.) These elements are given by the formula
$$ \sigma=\sum_{\sigma\leq w}\mathcal M_w,$$
which defines them recursively, via M\"obius inversion formula, since $\leq $ is an order.

The following result is equivalent to a result due to Aguiar and Sottile (\cite{AS} Theorem 3.1).

\begin{theorem}\label{AS}
For any permutation $\sigma$, one has 
$$\delta (\mathcal M_\sigma)=\sum_{\sigma=v\tr u}\mathcal M_u\otimes \mathcal M_v.$$
\end{theorem} 

We give below the proof of this result; it may help to understand the proof of the similar result on the Hopf algebra of tableaux, which we give in Section \ref{prim}. 

We begin by a lemma, that will also have an analogue for tableaux.

\begin{lemma}\label{sigma-a-b} Let $n=p+q$, $\sigma\in S_n$, $a=\sigma\mid\{1,\ldots,p\}$, $b=st(\sigma\mid\{p+1,\ldots,n\}$. 
Then, for $v\in S_q$, $u\in S_p$, $\sigma\leq v\tr u$ if and only if $a\leq u$ and $b\leq v$.
\end{lemma}

\begin{proof}
1. We show first that $\sigma\leq b\tr a$. We have $b\tr a=\overline{b}a$ (concatenation of words), where it is easily seen that $\overline{b}=\sigma\mid\{p+1,\ldots,n\}$ (add $p$ to each letter of 
$b$).

Let $(j,i)$ be an inversion of $\sigma$. Then $i<j$ and $j$ is at the left of $i$ in $\sigma$. If $j,i\leq p$, then $(j,i)$ 
is an inversion of $a$, and therefore also of $b\tr a$. If $j,i>p$, then $(j,i)$ is an inversion of $\overline{b}$, hence of $b\tr 
a$. If $i\leq p$ and $j>p$, then $(j,i)$ is an inversion of $b\tr a$, since in this latter permutation, each letter $>p$ is 
at the left of each letter $\leq p$. There is no other case since $j>i$. 

Thus $(j,i)$ is in each case an inversion of $b\tr a$, hence $\sigma\leq b\tr a$.

2. Suppose that $a\leq u$ and $b\leq v$. Then clearly  $b\tr a\leq v\tr u$, by Lemma \ref{perm-comp}. Hence by 1, $\sigma\leq v\tr u$.

3. Suppose that $\sigma\leq v\tr u$. If $(j,i)$ is an inversion of $a$, then it is an inversion of $\sigma$, hence of
$v\tr u=\bar vu$; therefore it is an inversion of $u$ since $i,j\leq p$; thus $a\leq u$. Similarly, $b\leq v$.
\end{proof}

\begin{proof}[Proof of Theorem \ref{AS}]

Define the $\Z$-linear mapping $\delta_1:\Z S\mapsto \Z S\otimes \Z S$ by $\delta_1(\mathcal M_\sigma)=\sum_{\sigma=v\tr u}\mathcal M_u\otimes \mathcal M_v$. It is enough to 
show that $\delta_1=\delta$. We have for any permutation $\sigma\in S_n$, 
$$
\delta_1(\sigma)=\delta_1(\sum_{\sigma\leq w}\mathcal M_w)=
\sum_{\sigma\leq w}\sum_{w=v\tr u}\mathcal M_u\otimes \mathcal 
M_v=
\sum_{\sigma\leq v\tr u}\mathcal M_u\otimes \mathcal 
M_v.$$ 
This is by Lemma \ref{sigma-a-b} equal to 
$$\sum_{0\leq i\leq n}
\sum_{\sigma\mid\{1,\ldots,i\}\leq u \atop
st(\sigma\mid\{i+1,\ldots,n\}\leq v}
\mathcal M_u\otimes \mathcal M_v$$
$$=\sum_{0\leq i\leq n}(\sum_{\sigma\mid\{1,\ldots,i\}
\leq u}\mathcal M_u)\otimes (\sum_{st(\sigma\mid\{i+1,\ldots,n\})\leq v} \mathcal M_v)$$
$$=\sum_{0\leq i\leq n}
\sigma\mid\{1,\ldots,i\} \otimes st(\sigma\mid\{i+1,\ldots,n\}=\delta(\sigma),$$ by the definition of the basis $\mathcal M_\sigma$.
\end{proof}

In order to understand the equivalence between the previous theorem and the theorem of Aguiar and Sottile, it 
may be useful to use the notion of global descents, introduced by them. Recall that according to Aguiar and Sottile (\cite{AS}, Definition 2.12), a {\em global 
descent} of $\sigma\in S_n$ is a position $i\in\{1,\ldots,n-1\}$ such that for any $j\leq i$ and any $k>i$ one has $
\sigma(j)>\sigma(k)$.

Then the permutations that are indecomposable for the product $\tr$ (which are the free generators of the free monoid $S$) are those which have no global descents. Moreover, if $\sigma=\sigma_1\tr\ldots\tr \sigma_k$ with 
indecomposable $\sigma_i$'s, then the global descents of $\sigma$ are the positions $|\sigma_1|, |\sigma_1|+|
\sigma_2|,\ldots,|\sigma_1|+\ldots |\sigma_{k-1}|$. 

For example, $78465213=12\tr 132 \tr 213$ has 2 and 5 as global descents, and $12,132,213$ are 
indecomposable for $\tr$, equivalently, have no global descents.

\begin{corollary} The submodule of primitive elements of $\Z S$ is spanned by the $\mathcal M_\sigma$ such that $\sigma$ has no global descent, or equivalently, $\sigma$ is indecomposable for the product $\tr$.
\end{corollary}

\section{Preliminaries on tableaux}

For unreferenced results quoted here, see \cite{S}. Denote by $T_n$ denotes the set of standard Young tableaux (we say simply {\em tableaux}) whose entries are $1,\ldots,n$. We denote by $(P(\sigma),Q(\sigma))$ the  pair of tableaux associated with $\sigma\in S_n$ by the Schensted correspondence. 

Let $T=\cup_{n\geq 0}T_n$ be the set of all standard tableaux. The {\em plactic equivalence} on $T$ is the smallest equivalence relation generated by the {\em Knuth relations} $xjiky\sim_{plax}xjkiy$, $xikjy\sim_{plax}xkijy$, $i<j<k$, for $i,j,k\in \mathbb{N}$ and $x,y \in \mathbb{N}^*$.

By Knuth's theorem, one has $P(\sigma)=P(\tau)$ if and only if $\sigma\sim_{plax}\tau$. Thus we may identify tableaux and plactic classes.
In the sequel, we use systematically this identification.

Following Taskin \cite{T} , we define the {\em weak order} of tableaux as follows: let $U,V\in T_n$; let $u,v\in S_n$ be such that 
$P(u)=U, P(v)=V$. Define the relation $U\leq V$ if $u\leq v$; then $\leq$ is the transitive closure of this relation. 

In other words, the weak order on tableaux is the smallest order on $T$ such  that the mapping $P:S\to T$ is increasing for the weak order.

This order was introduced by Melnikov \cite{M} (called there {\em Duflo order}) and Taskin \cite{T} (denoted there 
$\leq_{weak}$). The difficulty here is to show that it is indeed an order.

For two tableaux $A,U$, one has $A\leq U$ if and only if there exist $n\geq 1$ and permutations $\alpha_0,\ldots,\alpha_{n-1},\beta_1,\ldots,\beta_n$ such that 
\begin{equation}\label{<}
P(\alpha_0)=A, \alpha_0\leq \beta_1\sim\alpha_1\leq\beta_2\sim\ldots \alpha_{n-1}\leq\beta_n,P(\beta_n)=U.\end{equation}

The product $\tr$ of permutations induces a product of tableaux, still denoted by $\tr$. This follows from the next lemma.

\begin{lemma}\label{plax-comp} The plactic equivalence is compatible with the product $\tr$, that is: $u\sim_{plax}u' \Rightarrow v\tr u\sim_{plax} v\tr u'$.
\end{lemma}

\begin{proof} Suppose indeed that $u,u'\in S_p,v\in S_q$ and for some letters $1\leq i<j<k\leq p$, and words $x,y$, one 
has $u=xjiky$, $u'=xjkiy$. Then $\bar vu=\bar vxjiky,\bar vu'=\bar vxjkiy$ and therefore $v\tr u\sim_{plax}v\tr u'$. There are other 
similar cases, left to the reader.
\end{proof}

Therefore, one has for any permutations  $u,v$
$$
P(v\tr u)=P(v)\tr P(u),
$$
which means that $P$ is a homomorphism from the monoid $S$ into the monoid $T$, both with the product $\tr$.

Note that one may compute directly $V\tr U$ as follows: $V\tr U$ is the tableau obtained by letting 
fall $\bar V$ onto $U$, with $\bar V$ obtained by adding $p$ to each letter in $V$ (this follows from the dual Schensted correspondence, that is, column insertion). Thus $V\tr U$ is the tableau denoted $U/V$ in \cite{T}, p. 1109. 
For example:

\ytableausetup{aligntableaux=center}
$$U \ \ \begin{ytableau}
2\\
 1&  3 
 \end{ytableau}\ ,\ 
V \ \ \begin{ytableau}
*(blue!40) 1& *(blue!40) 2 \end{ytableau}
\ \ \ \  \rightarrow \ \ \ \ 
 {\begin{ytableau}
*(blue!40) 4& *(blue!40) 5 \end{ytableau}
\atop
\begin{ytableau}
2\\
 1&  3 
 \end{ytableau}}
 \ \ \ \ 
 \rightarrow\ \ \ \ 
 V\tr U \ \ 
 \begin{ytableau}
*(blue!40) 4\\
2& *(blue!40) 5\\
1&3 \end{ytableau}
$$
\bigskip

We need also the following.

\begin{lemma}\label{weak-comp} The weak order $\leq$ on tableaux is compatible with the product $\tr$: $U\leq U', V\leq V' \Rightarrow V\tr U\leq V'\tr U'$.
\end{lemma}

\begin{proof} This follows from Lemma 
\ref{perm-comp}, and the characterization through Eq.(\ref{<}) of the order, using the fact that $P$ is an increasing surjective homomorphims $S\to T$.
\end{proof}

By Proposition 2.5 in \cite{LS}, p.133, the plactic equivalence is compatible with the restriction to intervals, and with standardization. It follows 
that the plactic equivalence is compatible with the composition of the two operations: if $u,v\in S_n$, $u\sim_{plax}v$, 
and $I$ is an interval of $[n]$, then $st(u\mid I)\sim_{plax} st(v\mid I)$. Thus, if $A=P(u)$, we may denote without 
ambiguity by $st(A\mid I)$ the tableau $P(st(u\mid I))$.
By the work of Sch\"utzenberger, the corresponding tableau is obtained by jeu-de-taquin straightening of the skew 
tableau which is the restriction to $I$ of the tableau $P(u)$; but we do not need this fact.

\begin{lemma}\label{restr-order} Let $A,B\in T_n$ such that $A\leq B$, and $I$ be an interval in $[n]$. Then $st(A\mid I)\leq st(B\mid I)$.
\end{lemma}

\begin{proof} This follows from the characterization through Eq.(\ref{<}) of the order, Lemma \ref{restr-order-perm}, the previous observation, and the fact that $P$ is increasing.
\end{proof}

\section{Main lemma}

Recall the Taskin weak order on tableaux, denoted $\leq$.

\begin{lemma}\label{main}
Let $n=p+q$ and $\Sigma\in \mathcal T_n$. Let $A=st(\Sigma\mid\{1,\ldots,p\})$, and let $B=st(\Sigma\mid
\{p+1,\ldots,n\})$.
Then for $U\in \mathcal T_p$, $V\in \mathcal T_q$, one has: $\Sigma\leq V\tr U$ if and only if $A\leq U$ and $B\leq V$.
\end{lemma}

\begin{proof}[Proof of Lemma \ref{main}]

1. We show first that $\Sigma \leq B\tr A$. Let $\sigma \in S_n$ be such that $\Sigma=P(\sigma)$.
%Let $\sigma$ be the reading word of $\Sigma$. 
Let $a=st(\sigma\mid \{1,\ldots,p\})$ and $b=st(\sigma\mid \{p+1,\ldots,n\})$.
%as in Lemma \ref{sigma-a-b}. 
Then $\sigma \leq b\tr a$ by Lemma \ref{sigma-a-b}. 
It follows that $\Sigma \leq B\tr A$ by Lemmas \ref{plax-comp} and \ref{weak-comp}.

%We have $P(a)=A, P(b)=B$: indeed, the reading word of the tableau $A=\Sigma \mid \{1,\ldots,p\}$ is $\sigma \mid \{1,\ldots,p\}=a$; therefore, $P(a)=A$. Moreover, let $B'$ be the skew tableau $st(\Sigma \mid \{p+1,\ldots,n\})$. Then the reading word of $B'$ is $b$. Hence $P(b)$ is equal to the jeu-de-taquin straightening of $B'$, that is $B$.
%
%We deduce $\Sigma=P(\sigma)\leq P(b\tr a)$ (by definition of the order) $=P(b)\tr P(a)$ (Corollary \ref{P-hom}) $=B\tr A$.
%
2. Suppose that $A\leq U$ and $B\leq V$. Then by 1. and Lemma \ref{weak-comp}, we have $\Sigma\leq B\tr A\leq V\tr U$.

3. Suppose now that $\Sigma\leq V\tr U$. Let $u\in S_p,v\in S_q$ be such that $U=P(u),V=P(v)$.
Then by Lemma \ref{restr-order}, we have $A=st(\Sigma\mid \{1,\ldots,p\})
\leq st((V\tr U)\mid \{1,\ldots,p\})=P(st((v\tr u)\mid\{1,\ldots,p\}))=P(u)=U$. 
Moreover, by the same lemma, $$B=st(\Sigma\mid \{p+1,\ldots,n\})\leq st((V\tr U)\mid \{p+1,\ldots,n\})
$$
$$=P(st((v\tr u)\mid\{p+1,\ldots,n\}))=P(v)=V.$$
\end{proof}

\section{Primitive elements in the Hopf algebra of tableaux}\label{prim}

The free $\Z$-module $\Z T$, based on the set $T$ of tableaux, becomes a structure of Hopf algebra, quotient 
of the Hopf algebra $\Z S$ of Section \ref{ZS}, and whose product and coproduct are therefore also denoted by $*$ and $\delta$. The quotient is obtained by identifying plactic equivalent permutations. 
In other words, consider the submodule $I$ spanned by the elements $u-v$, $u\sim_{plax} v$; then $I$ is an ideal and 
a co-ideal of $\Z S$, and the quotient $\Z S/I$ is canonically isomorphic with $\Z T$. Moreover, the canonical bialgebra homomorphism $\Z S\to\Z T$ maps each permutation $\sigma$ onto $P(\sigma)$. See \cite{PR}, Th\'eor\`eme 3.4 and 4.3 (iv), where the product 
and coproduct are there denoted $*'$ and $\delta'$.

%According to \cite{PR} 5.b p.85, the coproduct $\delta$ is directly defined on tableaux as follows: 
%for $\Sigma\in \mathcal T_n$, $\delta(\Sigma)$ is the sum, over all $p+q=n$, of all $A\otimes B$ obtained as in Lemma 
%\ref{main}.
Now we introduce a new basis of $\mathbb ZT$, following the method of Aguiar and Sottile \cite{AS}, replacing 
the weak order on permutations by the Taskin weak order on tableaux. The new basis (that we may call the {\em monomial basis}, following \cite{AS}) $\mathcal M_W$, $W\in T$, is completely defined by the identities

$$ \Sigma=\sum_{\Sigma\leq W}\mathcal M_W,$$

for all tableau $\Sigma$, via M\"obius inversion on the poset $(T,\leq)$.

\begin{theorem}\label{coproduct} Let $\Sigma\in \mathcal T_n$. Then
$$\delta (\mathcal M_\Sigma)=\sum_{\Sigma=V\tr U}\mathcal M_U\otimes \mathcal M_V.$$
\end{theorem}

\begin{proof}
Define the $\Z$-linear mapping $\delta_1:\Z T\mapsto \Z T\otimes \Z T$ by $$\delta_1(\mathcal M_W)=\sum_{W=V\tr U}\mathcal M_U\otimes \mathcal M_V.$$ It is 
enough to show that $\delta_1=\delta$.

We have $$\delta_1(\Sigma)=\delta_1(\sum_{\Sigma\leq W}\mathcal M_W)=\sum_{\Sigma\leq W}\sum_{W=V\tr U}
\mathcal M_U\otimes \mathcal M_V=\sum_{\Sigma\leq V\tr U}\mathcal M_U\otimes \mathcal  M_V$$
$$=\sum_{p+q=n}
\sum_{\Sigma\leq V\tr U \atop V\in \mathcal T_q,U\in \mathcal T_p}\mathcal M_U\otimes \mathcal  M_V. $$

This is by Lemma \ref{main}, and with its notations, equal to 
$$\sum_{p+q=n}
(\sum_{st(\Sigma\mid\{1,\ldots,p\})\leq U}\mathcal  M_U)
\otimes 
(\sum_{st(\Sigma\mid\{p+1,\ldots,n\})\leq V}\mathcal M_V)
$$
$$=\sum_{p+q=n}(st(\Sigma\mid\{1,\ldots,p\}))
\otimes (st(\Sigma\mid\{p+1,\ldots,n\})).$$
Choose $\sigma$ such that $P(\sigma)=\Sigma$. Then by definition of $st(\Sigma\mid I)$, the latter quantity is equal to
$$
\sum_{p+q=n}P(st(\sigma\mid\{1,\ldots,p\}))
\otimes P(st(\Sigma\mid\{p+1,\ldots,n\}))$$
$$=(P\otimes P)(\sum_{p+q=n}st(\sigma\mid\{1,\ldots,p\})
\otimes st(\sigma\mid\{p+1,\ldots,n\}))$$
$$=(P\otimes P)(\delta(\sigma))=\delta(P(\sigma)
=\delta(\Sigma)),$$
since $P$ is a homomorphism of bialgebra, as recalled at the beginning of Section \ref{prim}.
\end{proof}

\begin{corollary} The submodule of primitive elements of $\Z T$ is spanned by the $\mathcal M_\Sigma$ sucht that $\Sigma$ is indecomposable for the product $\tr$.
\end{corollary}

The dimensions of the graded components of the submodule of primitive elements is therefore the sequence of the numbers of tableaux indecomposable for the product $\tr$; it is 
denoted $a_n$ in \cite{PR}, p.88-89. For $n=1,2,\ldots,10$, they are the numbers
$$
1,1,1,3,7,23,71,255, 911,3535.
$$
They appear as sequence
A140456 in the Online Encyclopedia of Integer Sequences \cite{OEIS}, with other interpretations.

\section{Further remarks}

\subsection{Product formulas using $\Delta$}

The product $\tr$, both for permutations and tableaux, plays a role in product formulas in the dual Hopf algebras of $\Z S$ and $\Z T$.

First, one has to consider also the product $\square$ of permutations: let $a\in S_p$, $b\in S_q$; then $a\square b=a\bar b$, where $\bar b$ is obtained from $b$ by adding $p$ to each letter in $b$.

Recall the {\em shifted shuffle product} of permutations, denoted $\overline{\shuffle}$: $a \overline{\shuffle} b$ is the shuffle of $a$ and $\bar b$. This product is the dual product of the coproduct $\delta$ of $\Z S$. On has

\begin{theorem}(Loday-Ronco \cite{LR} Theorem 4.1)
Let $a,b\in S$. Then
$$
a \overline{\shuffle} b=\sum_{a\square b\leq \sigma\leq b\tr a} \sigma.
$$
\end{theorem}

In other words, $a \overline{\shuffle} b$ is the sum of all permutations in the interval $[a\square b, b\tr a]$ of the weak order.

The product $\square$ is compatible with the plactic equivalence, hence $\square$ is well-defined on tableaux. Denote also by $\overline{\shuffle}$ the product which is the dual of $\delta$ in the dual coalgebra of $\Z T$. Then one has

\begin{theorem}(Taskin \cite{T} Theorem 4.1)
Let $A,B\in \mathcal T$. Then
$$
A\overline{\shuffle} B=\sum_{A\square B\leq T\leq B\tr A} T.
$$
\end{theorem}
In other words, $A\overline{\shuffle} B$ is the sum of all tableaux belonging to  the interval $[A\square B, B\tr A]$ of the weak order. Note that the product $A\square B$ is denoted $A\backslash B$ in \cite{T}, p. 1109. It corresponds to put the tableau $\overline{B}$ aside the tableau $A$, then push every row of $\overline{B}$ toward $A$. For example: 
\bigskip

%\ytableausetup{aligntableaux=center}
$A \ \ \ytableausetup{aligntableaux=bottom}
\begin{ytableau}
2\\
 1&  3 
 \end{ytableau}\ ,\ 
B \ \ \begin{ytableau}
*(blue!40) 3\\
*(blue!40) 2\\
*(blue!40)1& *(blue!40) 4 \end{ytableau}\ \ 
\rightarrow$ \ \ 
 $\ytableausetup{aligntableaux=bottom}
\begin{ytableau}
2\\
 1&  3 
 \end{ytableau} \ \ 
\begin{ytableau}
*(blue!40) 6\\
*(blue!40) 5\\
*(blue!40)4& *(blue!40) 7 \end{ytableau}$\ \ 
 $\ytableausetup{aligntableaux=bottom}
 \rightarrow\ \ 
 A\square B \ \ 
 \begin{ytableau}
*(blue!40) 6\\
2& *(blue!40) 5\\
1&3 & *(blue!40) 4 &*(blue!40) 7\\  \end{ytableau}
$
\bigskip

\subsection{Multiplicative basis}

Note that Theorem \ref{coproduct} means that in the dual Hopf algebra of $\Z T$, the dual basis $\mathcal M_T^*$ of the basis $\mathcal M_T$ is {\em multiplicative}, in the sense of \cite{DHNT}: one has for tableaux $A,B$, $$\mathcal M_A^*\overline{\shuffle}\mathcal M_B^*=\mathcal M_{B\tr A}^*.$$

Indeed, using the canonical pairing between $\Z T$ and its dual, $$\l \mathcal M_A^*\overline{\shuffle}\mathcal M_B^*,\mathcal 
M_\Sigma\r=\l\mathcal M_A^* \otimes \mathcal M_B^*,\delta(\mathcal M_\Sigma)\r=\l \mathcal M_A^* \otimes \mathcal 
M_B^*,\sum_{\Sigma=V\tr U}\mathcal M_U\otimes \mathcal M_V\r$$
$$=\sum_{\Sigma=V\tr U}\l \mathcal M_A^*,\mathcal M_U\r\l\mathcal M_B^*,
\mathcal M_V\r. $$

This is 1 exactly when $\Sigma=B\tr A$, otherwise it is 0. Thus it is equal to $\l \mathcal M_{B\tr A}^*,\mathcal M_\Sigma\r$, which proves the formula.

\subsection{A counter-example by Franco Saliola}

Our basis $\mathcal M_\Sigma$ was inspired by the construction of Aguiar and Sottile in \cite{AS}. They prove further 
that the structure constants for multiplication are positive (Theorem 4.1 in their article). This is not true in the case of Poirier--Reutenauer Hopf algebra of tableaux, as shows 
a counter-example of Franco Saliola, that he kindly permitted us to reproduce here. Indeed, he computed that
$$
\mathcal M_{P(123)}*\mathcal M_{P(123)}=\mathcal M_{P(123456)}$$
$$-\mathcal M_{P(241356)}-\mathcal M_{P(251346)} -\mathcal M_{P(261345)}-\mathcal M_{P(351236)}-\mathcal M_{P(361245)}-\mathcal M_{P(461235)}$$
$$+\mathcal M_{P(256134)}+\mathcal M_{P(346125)}+\mathcal M_{P(356124)} +2\mathcal M_{P(456123)}$$
$$+2\mathcal M_{P(362514)}-\mathcal M_{P(462513)} \mathcal -M_{P(543126)}.
$$

\medskip

{\em Acknowledgments}. We thank Franco Saliola, who allowed us to include a counter-example arising from his 
computations.

\end{document}